\documentclass[12pt,oneside]{article}
\usepackage{amsfonts}
\usepackage{amsmath}
\usepackage{amssymb}
\usepackage{amsthm}
\usepackage{enumerate}
\usepackage[latin1]{inputenc} 
\usepackage{color}
\usepackage{mathtools}
\usepackage{pifont}
\usepackage{hyperref}

\newtheorem{theorem}{Theorem}%[section]

\newtheorem{lemma}[theorem]{Lemma}

\theoremstyle{definition}
\newtheorem{example}[theorem]{Example}

\DeclareMathOperator{\tr}{tr}

\usepackage[paperwidth=210mm,paperheight=300mm,
textwidth=170mm,textheight=250mm,top=25mm,lmargin=20mm%,rmargin=10mm,
]{geometry}

\title{Fillmore Theorem for integers
\footnote{{\bf Keywords:} Inverse problem, field, similarity, diagonal.
}
\footnote{{\bf Mathematics subject classification:} 15A83.}
 \footnote{Supported by the Spanish Ministerio de Ciencia y Tecnología MTM2015-68805-REDT.}}
\author{Alberto Borobia\\
\small Dpto. Matem\'{a}ticas, Universidad Nacional de Educaci\'on a Distancia (UNED), Spain\\
\small e-mail: $aborobia@mat.uned.es$ \\ \\
\small \url{}}
\date{}

\begin{document}

\maketitle

\begin{abstract}
Fillmore Theorem  says   that if $A$ is a nonscalar matrix  of order $n$ over a field $\mathbb{F}$ and  $\gamma_1,\cdots,\gamma_n\in \mathbb{F}$ are such that $\gamma_1+\cdots+\gamma_n=\tr A$, then there is a matrix $B$ similar to $A$ with diagonal $(\gamma_1,\cdots,\gamma_n)$. Fillmore proof works by induction  on the size of $A$ and implicitly provides an algorithm to construct   $B$.  We  develop an   explicit  and extremely simple  algorithm  that finish in two steps (two similarities), and with its help we  extend Fillmore Theorem to integers  (if $A$ is integer then we can require to $B$  to be  integer). 
\end{abstract}

\section{Introduction}

An \emph{inverse  problem} asks for the existence of a matrix with prescribed structural and spectral constraints. The following is an early inverse problem result  stated by Fillmore in 1969.

\begin{theorem}[]\label{FillmoreTheorem}
Let $A$ be a nonscalar matrix  of order $n$ over a field $\mathbb{F}$ and let   $\gamma_1,\ldots,\gamma_n\in  \mathbb{F}$ such that $\gamma_1+\cdots+\gamma_n= \tr A$. Then there is a matrix similar to $A$ with diagonal $(\gamma_1,\ldots,\gamma_n)$. 
\end{theorem}

The proof given in~\cite{Fillmore}  is by induction on the size of $A$ and implicitly provides  an   algorithm   to construct a matrix  similar to $A$ with   diagonal $(\gamma_1,\ldots,\gamma_n)$.  Tough the algorithm is elementary,   it requires  some tedious calculus for each induction step.  For completeness we will include this proof,  remaining as  faithful as possible to Fillmore presentation. As the original proof  has some inaccuracy then  we will incorporate  some modifications taken from Zhan~\cite[Theorem 1.5]{Zhan}.

\begin{lemma}\label{Auxiliar}
Let $A$ be a nonscalar matrix of order $n\geq 3$ over a field $\mathbb{F}$  and let $\gamma \in \mathbb{F}$. Then there is a nonsingular $P$ such that 
$
PAP^{-1}=\begin{bmatrix}
\gamma & * \\ 
* & A_1
\end{bmatrix}
$
where  $A_1$ is  nonscalar  of order $n-1$.
\end{lemma}

\begin{proof}
Since $A$ is nonscalar, there is a vector $x$ such that $x$ and $Ax$ are linearly independent.  Such a vector $x$ can be taken to be a standard basis vector if $A$ is not diagonal, or the sum of a pair of standard basis vectors otherwise.  Let $\{x,Ax-\gamma x,x_3,\ldots,x_n\}$ be a basis of $\mathbb{F}^{n}$. If $(\alpha_{ij})_{i,j=1}^n$ is the matrix of $A$ in this basis, then $\alpha_{11}=\gamma$, $\alpha_{21}=1$ and $\alpha_{31}=\cdots=\alpha_{n1}=0$. Let  $P=(p_{ij})_{i,j=1}^n$ be equal to the  identity except on its entry $p_{13}=1-\alpha_{23}$. Then $PAP^{-1}$ has $\gamma$ on its $(1,1)$ entry  and $1$ on its $(2,3)$ entry.
\end{proof}

\noindent {\bf Proof of  Fillmore Theorem:} 
The proof is by induction on the size of $A$. The conclusion of Lemma~\ref{Auxiliar} allows it to be applied repeatedly, leaving the case $n=2$ of the theorem for consideration. Again let $x$ be a vector such that $x$ and $Ax$ are linearly independent. Then $\{x,Ax-\gamma x\}$ form a basis, and in this basis the diagonal of $A$ is  
$(\gamma,\tr A -\gamma)$. 
$\square$

\bigskip

\section{Alternative algorithm}\label{AlternativeProcedure}

We present a  two steps algorithm  that minimizes the required computation. Namely, it starts with $A$ and performs two  similarities to reach a matrix with diagonal $(\gamma_1,\ldots,\gamma_n)$. On each step the matrix that performs the similarity differs from the identity by one line (row or column). Some  results, which can be demonstrated by routine check, are necessary.

\begin{lemma}\label{DiagonalMatrix}
Let $A=(a_{ij})_{i,j=1}^{n}$ be a nonscalar diagonal matrix  over a field $\mathbb{F}$, let $s$  such that  $ a_{11}\neq a_{ss}$, and let $B=(b_{ij})_{i,j=1}^{n}$ be equal to the  identity except on its entry   $b_{1s}=\frac{1}{a_{ss}-a_{11}}$. Then  $BAB^{-1}$ is equal to $A$ except on its entry $(1,s)$ that is equal to 1. 
\end{lemma}

\begin{lemma}\label{OffDiagonalNonzero}
Let $A=(a_{ij})_{i,j=1}^{n}$ be a non-diagonal matrix  over a field $\mathbb{F}$,  let  $r\neq s$ such that  $a_{rs}\neq 0$, and let $B=(b_{ij})_{i,j=1}^{n}$ be equal to the  identity except on its row $s$ where  $b_{sr}=0$, $b_{ss}=a_{rs}$, and $b_{sk}=a_{rk}-1$ for $k\notin\{r,s\}$. Then all off-diagonal entries of the row $r$ of $BAB^{-1}$ are equal to 1.
\end{lemma}

\begin{lemma}\label{OffDiagonalOnes}
Let $A=(a_{ij})_{i,j=1}^{n}$ be a matrix  over a field $\mathbb{F}$ such that for some $r$ the off-diagonal entries of row $r$ are equal to 1,  let $\gamma_1,\ldots,\gamma_n\in \mathbb{F}$ with $\gamma_1+\cdots+\gamma_n=\tr A$, and let $B=(b_{ij})_{i,j=1}^{n}$ be equal to the  identity except on its column $r$ where  $b_{rr}=1$, and $b_{kr}=\gamma_k-a_{kk}$ for $k\neq r$. Then  the diagonal of $BAB^{-1}$ is  $(\gamma_1,\ldots,\gamma_n)$. 
\end{lemma}

\medskip 

\noindent {\bf Proof of   Fillmore Theorem:}  If $A$ is  nondiagonal  then we construct a matrix similar to $A$  with  diagonal  $(\gamma_1,\ldots,\gamma_n)$ by the successive application of Lemma~\ref{OffDiagonalNonzero} and Lemma~\ref{OffDiagonalOnes}. If $A$ is diagonal  nonscalar then first  apply Lemma~\ref{DiagonalMatrix} to obtain a nondiagonal matrix.  $\square$

\begin{example} \label{Example}
We wish to construct a matrix with diagonal $(3,5,-2,6,-1)$ and similar to 
$$A=
\begin{pmatrix}
 4 & 0 & 4 & -3 & 5 \\
 2 & 3 & 0 & 2 & 3 \\
 0 & -2 & 2 & 5 & 4 \\
 7 & 1 & 3 & 4 & 0 \\
 2 & 5 & 3 & 0 & -2 \\
\end{pmatrix}
$$
The trace condition is satisfied. We  apply Lemma~\ref{OffDiagonalNonzero} to $A$ with respect to  $a_{34}=5$, so 
$$B_1=
\begin{pmatrix}
 1 & 0 & 0 & 0 & 0 \\
 0 & 1 & 0 & 0 & 0 \\
 0 & 0 & 1 & 0 & 0 \\
 -1 & -3 & 0 & 5 & 3 \\
 0 & 0 & 0 & 0 & 1 \\
\end{pmatrix}
\Longrightarrow B_1AB_1^{-1}=
\begin{pmatrix}
 \frac{17}{5} & -\frac{9}{5} & 4 & -\frac{3}{5} & \frac{34}{5} \\
 \frac{12}{5} & \frac{21}{5} & 0 & \frac{2}{5} & \frac{9}{5} \\
 1 & 1 & 2 & 1 & 1 \\
 \frac{172}{5} & \frac{106}{5} & 20 & \frac{17}{5} & -\frac{151}{5} \\
 2 & 5 & 3 & 0 & -2 \\
\end{pmatrix}
$$
And we apply Lemma~\ref{OffDiagonalOnes} to $B_1AB_1^{-1}$ with respect to its third row, so 
$$B_2=
\begin{pmatrix}
 1 & 0 & -\frac{2}{5} & 0 & 0 \\
 0 & 1 & \frac{4}{5} & 0 & 0 \\
 0 & 0 & 1 & 0 & 0 \\
 0 & 0 & \frac{13}{5} & 1 & 0 \\
 0 & 0 & 1 & 0 & 1 \\
\end{pmatrix}
\Longrightarrow B_2(B_1AB_1^{-1})B_2^{-1}=
\begin{pmatrix}
 3 & -\frac{11}{5} & \frac{59}{25} & -1 & \frac{32}{5} \\
 \frac{16}{5} & 5 & -\frac{171}{25} & \frac{6}{5} & \frac{13}{5} \\
 1 & 1 & -2 & 1 & 1 \\
 37 & \frac{119}{5} & \frac{824}{25} & 6 & -\frac{138}{5} \\
 3 & 6 & -\frac{1}{5} & 1 & -1 \\
\end{pmatrix} \quad \square
$$

\end{example}

\section{Fillmore Theorem for integers}\label{AlternativeProcedure}

In Example~\ref{Example}, if we start applying Lemma~\ref{OffDiagonalNonzero} to $A$ with respect to  $a_{42}=1$   and after that  we apply Lemma~\ref{OffDiagonalOnes} to the resulting matrix with respect to its fourth  row we obtain the sequence
$$
A\Longrightarrow 
\begin{pmatrix}
 4 & 0 & 4 & -3 & 5 \\
 60 & -6 & 37 & -6 & 37 \\
 12 & -2 & 6 & 5 & 2 \\
 1 & 1 & 1 & 4 & 1 \\
 -28 & 5 & -7 & 0 & 3 \\
\end{pmatrix}
\Longrightarrow 
\begin{pmatrix}
 3 & -1 & 3 & 47 & 4 \\
 71 & 5 & 48 & 630 & 48 \\
 4 & -10 & -2 & 47 & -6 \\
 1 & 1 & 1 & 6 & 1 \\
 -32 & 1 & -11 & -151 & -1 \\
\end{pmatrix}
$$
This last matrix is integer, similar to $A$, and has diagonal $(3,5,-2,6,-1)$. In order to  obtain an integer matrix  it was important that $A$ had an off-diagonal entry equal to 1.

\begin{lemma}\label{SimilarToEntry1}
Let $A=(a_{ij})_{i,j=1}^{n}$ be a nonscalar integer matrix. Then there is an integer matrix  similar to $A$ with  an off-diagonal entry equal to $1$.
\end{lemma}

\begin{proof}
If $A$ is  diagonal then the result follows from Lemma~\ref{DiagonalMatrix}, and if $A$ is nondiagonal and has a nonzero entry equal to $1$ then there is nothing to prove. So suppose that  $A=(a_{ij})_{i,j=1}^{n}$ is  nondiagonal  with none of its off-diagonal entries equal to $1$. Without loss of generality assume that $a_{12}\neq0$. We develop a simple algorithm:

\begin{description}
\item[Step 1.] If $(a_{13},\ldots,a_{1n})\neq(0,\ldots,0)$ then go to Step 2. Otherwise go to Step 3.
\item[Step 2.]  Let $k$ be the minimum integer of $\{3,4,\ldots,n\}$ such that $a_{1k}\neq 0$. Let $m=gcd(a_{12},a_{1k})$,   $p=\frac{a_{12}}{m}$, and $q=\frac{a_{1k}}{m}$. As $p$ and $q$ are coprimes then let  $r$ and $s$ two integers such that $ps-qr=1$.   Let  $B=(b_{ij})_{i,j=1}^{n}$ be equal to the  identity except on its entries $b_{22}=p$, $b_{2k}=q$, $b_{k2}=r$, and $b_{kk}=s$. Then  $BAB^{-1}$ is integer and its first row is equal to $(a_{11},m,0,\ldots,0,a_{1,k+1},\ldots,a_{1n})$. If $m=1$ then we have finished. Otherwise do $A:=BAB^{-1}$ and go to Step 1.
\item[Step 3.] Let $B=(b_{ij})_{i,j=1}^{n}$ be equal to the  identity except on its entry  $b_{11}=1/a_{12}$. Then $BAB^{-1}$ is an integer matrix with first row $(a_{11},1,0,\ldots,0)$. And we have finished. 
\end{description}
\end{proof}

\begin{theorem}\label{FillmoreTheoremForIntegers}
Let $A$ be a nonscalar integer matrix  of order $n$  and let   $\gamma_1,\ldots,\gamma_n\in  \mathbb{Z}$ such that $\gamma_1+\cdots+\gamma_n= \tr A$. Then there is an integer matrix similar to $A$ with diagonal $(\gamma_1,\cdots,\gamma_n)$. \end{theorem}

\begin{proof}
By Lemma~\ref{SimilarToEntry1} $A$ is similar to an integer matrix $A_1$  with  an off-diagonal $(r,s)$ entry  equal to 1. Applying   Lemma~\ref{OffDiagonalNonzero} to $A_1$ with respect to its $(r,s)$ entry  we obtain  a similar matrix $B_1A_1B_1^{-1}$ which is also integer (since $B_1$ is integer and $\det B_1=1$) and has  the off-diagonal entries of its row $r$ equal to 1. Applying  Lemma~\ref{OffDiagonalOnes} to $B_1A_1B_1^{-1}$   with respect to its  row $r$  then we obtain a  similar matrix $B_2(B_1A_1B_1^{-1})B_2^{-1}$ which is also integer (since $B_2$ is integer and $\det B_2=1$) and has diagonal $(\gamma_1,\cdots,\gamma_n)$.
\end{proof}

\bibliographystyle{plain}

\begin{thebibliography}{1}

%\bibitem{Carlen-Lieb}
%E. Carlen and E. Lieb,
%Short proofs of theorems of Mirsky and Horn on diagonals and eigenvalues of matrices,
%Electron. J. Linear Algebra 18 (2009) 438-441.

%\bibitem{Chugunov}
%V. N. Chugunov,
%On the transformation of a matrix into a matrix with a given principal diagonal via a similarity transformation.
%Applied Programs Packages [in Russian], Moscow State University Press, Moscow (1997) 133--140.

%\bibitem{Dokovic}
%D. Dokovi\'{c},
%Generalization of Mirsky's theorem on diagonals and eigenvalues of matrices,
%Linear Algebra Appl. 437 (2012) 2680-2682.

%\bibitem{Farahat-Ledermann}
%H. K. Farahat and W. Ledermann,
%Matrices with prescribed characteristic polynomials,
%Proc. Edinburgh Math. Soc. 11 (1958/59) 143-146.

\bibitem{Fillmore}
P. A. Fillmore,
On similarity and the diagonal of a matrix,
Amer. Math. Monthly  76 (1969) 167--169.

%\bibitem{Gibson}
%P. M. Gibson,
%Matrix commutators over an algebraically closed field,
%Proc. Amer. Math. Soc. 52 (1975) 30--32.


%\bibitem{Horn-Johnson}
%R. Horn and C. Johnson, 
%Matrix Analysis, 2nd ed.
%Cambridge University Press, Cambridge, 2012.

%\bibitem{Ikramov-Chugunov}
%Kh.D. Ikramov,  V.N. Chugunov, 
%Inverse matrix eigenvalue problems,
%J. Math. Sci. (New York) 98(1) (2000) 51--136.

%\bibitem{Johnson-Li-Rodman-Spitkovsky-Pierce}
%C. R. Johnson, C.-K. Li, L. Rodman, I. Spitkovsky, and S. Pierce,
%Linear maps preserving regional eigenvalue location,
%Linear Multilinear Algebra 32  (1992) 253--264.

%\bibitem{Johnson-Shapiro}
%C. R. Johnson, H. Shapiro,
%Mathematical aspects of the relative gain array $(A\cdot A^{-T})^*$,
%SIAM J. Alg. Discr. Math. 7 (4)  (1986) 627--644.

%\bibitem{London90}
%D. London, 
%On matrices stochastically similar to matrices with equal diagonal elements, 
%Linear Multilinear Algebra 26 (1990) 107?117.
%
%\bibitem{London95}
%D. London,
%Diagonals of matrices stochastically similar to a given matrix,
%Linear Alg. Appl. 214 (1995) 1--10.
%
%\bibitem{London96}
%D. London, 
%On diagonals of matrices doubly stochastically similar to a given matrix, 
%Linear Algebra Appl., 244 (1996) 305--340.

%\bibitem{Mirsky}
%L. Mirsky, 
%Matrices with prescribed characteristics roots and diagonal elements,
%J. London Math. Soc. 33 (1958) 14--21.

\bibitem{Zhan}
X. Zhan, 
Matrix Theory, 
Graduate Studies in Mathematics 147, 
American Mathematical Society, 2013.


\end{thebibliography}

\end{document}